\documentclass[reqno,11pt]{amsart}

\usepackage{amssymb, amsmath,latexsym,amsfonts,amsbsy, amsthm,mathtools, stmaryrd}
\usepackage{rotating}
\usepackage{tikz}
\usepackage{calc}
\usepackage{tikz-cd}
\usepackage{mathabx}
\usepackage{tikz-3dplot}
\usetikzlibrary{calc}
\usetikzlibrary{decorations.markings,arrows}
\usetikzlibrary{shapes.geometric}
\usepackage{color}
\usepackage{extarrows}
\usepackage{graphicx}
\allowdisplaybreaks[4]
\let\ve=\varepsilon

\newcommand{\beq}{\begin{equation}}
\newcommand{\eeq}{\end{equation}}
\newcommand{\ben}{\begin{eqnarray}}
\newcommand{\een}{\end{eqnarray}}
\newcommand{\beno}{\begin{eqnarray*}}
\newcommand{\eeno}{\end{eqnarray*}}

\renewcommand{\theequation}{\thesection.\arabic{equation}}

\newtheorem{theorem}{Theorem}[section]

\newtheorem{proposition}[theorem]{Proposition}

\newtheorem{Theorem}{Theorem}[section]

\newtheorem{Proposition}[Theorem]{Proposition}
\newtheorem{Lemma}[Theorem]{Lemma}

\newtheorem{Remark}[Theorem]{Remark}

\newcommand{\ud}{\mathrm{d}}

\newcommand{\nn}{\mathbf{n}}

\newcommand{\vv}{\mathbf{v}}

\newcommand{\BB}{\mathbf{B}}

\newcommand{\EE}{\mathbf{E}}

\newcommand{\II}{\mathbf{I}}

\newcommand{\QQ}{\mathbf{Q}}
\newcommand{\PP}{\mathbf{P}}

\definecolor{grey}{rgb}{0.5,0.5,0.5}

\begin{document}
\title[Anisotropic Isotropic-Nematic]{Asymptotic limit for the isotropic-nematic problem with anisotropic elasticity}

\author{Fanghua Lin}
\address{Courant Institute of Mathematical Science, New York University}
\email{linf@cims.nyu.edu}

\author{Wei Wang}
\address{School of Mathematical Sciences, Zhejiang University, Hangzhou 310027, China}
\email{wangw07@zju.edu.cn}

\author{Zhifei Zhang}
\address{School of Mathematical Sciences, Peking University, Beijing 100871, China}
\email{zfzhang@math.pku.edu.cn}

\renewcommand{\theequation}{\thesection.\arabic{equation}}
\setcounter{equation}{0}


\begin{abstract}
We consider the isotropic-nematic interface problem for liquid crystals based on the Landau-de Gennes model.
We show that when the anisotropic elastic constant $L_2$ is negative, then the surface tension strength of the isotropic-nematic interface is proportional to  $\sqrt{L_1+\frac{2}{3}{L}_2}$, and the homeotropic alignment is more favored near the interface. This result gives a rigorous confirmation for de Gennes' formal derivation in \cite{deGen1971} on the isotropic-nematic tension strength and the anchoring alignment condition in the case of $L_2<0$.
\end{abstract}

\date{\today}
\maketitle

\numberwithin{equation}{section}

\section{Introduction}
\subsection{Background of the isotropic-nematic phase transition problem}

Phase transition problems have drawn great interest in both analysis and numerical simulations due to the wide applications in various physical systems and mathematical challenges. A commonly used model for phase transition is the following Cahn-Hilliard functional:
\begin{align}\label{intro:scalar}
  E_\ve(u)=\int_{\Omega}\Big(\frac{1}{2}|\nabla u|^2+\frac{1}{\ve^2}F(u)\Big)\ud x,
\end{align}
where $\ve$ is a small positive parameter, $\Omega\subset\mathbb{R}^n$ is a bounded domain, and $u:\Omega\to\mathbb{R}^k$ is a phase-indicator function. The potential function $F:\mathbb{R}^k\to\mathbb{R}_{\ge 0}$ vanishes at several isolated points or on several disjoint manifolds, which represent different phases or equilibrium states. There have been many works devoting to the
asymptotic behaviour of minimizers for \eqref{intro:scalar} when $k=1$ (i.e., $u$ is a scalar function), see \cite{ModiMo, Modi, Ster, KS, FT} for examples. When $k=2$ and the potential wells of $F$ are points or curves, the problem is studied by Fonseca-Tartar \cite{FT}, Sternberg \cite{Ster}, and Andre-Shafrir \cite{AS}. More recently, Lin-Pan-Wang \cite{LPW} considered the general case of $k\ge 2$ and rigorouly analyzed the leading order and next order in the expansion of the energy at minimizers.

Liquid crystal are a kind of material which has physical properties between the conventional fluids and solids (crystals). They may flow like ordinary fluid, but the molecules tend to have the same alignment near a material point, thus showing anisotropy. This distinguishing property makes liquid crystal materials have a wide range of applications. There are many different phases in liquid crystals, including the isotropic, nematic, and smectic phases. In nematic phase, the molecules tend to align in a common direction but are disorder in translation.

There are many different mathematical models to characterize the anisotropic structure of liquid crystals. One of the most commonly used model is the
Landau-de Gennes $\QQ$-tensor theory. In this model,  at each material point $x$, it is introduced a matrix
\begin{align*}
 \QQ(x)\in \mathcal{Q}:=\{\QQ\in\mathbb{R}^{3\times 3}:\quad  \QQ =\QQ^T,~ \mathrm{tr}\, \QQ=0 \}.
\end{align*}
The space $\mathcal{Q}$ is called the space of $\QQ$-tensor. If $\QQ(x)$ has three equal eigenvalues, then $\QQ(x)$ must be $0$ which is called {\it isotropic};  If $\QQ(x)$ has two equal eigenvalues which differ from the third one, we call $\QQ$ is uniaxial; If all eigenvalues of $\QQ(x)$ are different, we say $\QQ$ is biaxial.
For uniaxial $\QQ(x)$, one may write
\begin{align*}
  \QQ=s(\nn\otimes\nn-\frac13\II),\qquad s\in\mathbb{R},\quad \nn\in\mathbb{S}^2,
\end{align*}
where $\II$ is the $3\times 3$ identity matrix. In this case, $s$ is called the order parameter and $\nn$ is interpreted as the alignment of liquid crystal molecules.

The Landau de-Gennes energy for liquid crystals can be written as
\begin{align}\label{energy:LdG}
  \mathcal{F}_\ve(\QQ,\nabla\QQ)=\int_{\Omega}\Big(F_e(\nabla\QQ)+\frac{1}{\ve^2}F_b(\QQ)\Big)\ud x.
\end{align}
The first term $F_e(\nabla \QQ)$ is the elastic energy which takes the form
\begin{align}\label{energy:LdG-elastic}
F_e(\nabla\QQ)=&\frac{L_1}{2}|\nabla \QQ|^2+\frac{L_2}{2}Q_{ij,j}Q_{ik,k}.
\end{align}
Here $L_1, L_2$ are elastic coefficients, and in particular, the $L_2$ term represents the anisotropy of elasticity, which is a main feature for liquid crystal materials.  The second term $F_b$, called the bulk energy, are usually chosen as the following polynomial form:
\begin{align}\label{energy:LdG-bulk}
F_b(\QQ)=&\frac{a}{2}|\QQ|^2-\frac{b}{3}\mathrm{tr}\,\QQ^3+\frac{c}{4}(|\QQ|^2)^2,
\end{align}
where $a, b, c$ are material-dependent constants.  The bulk energy $F_b$ has two minimal wells at
\begin{align*}
\{0\} \text{ and }\mathcal{N}:=\Big\{s_+(\nn\nn-\frac13\II):\quad \nn\in\mathbb{S}^2, s_+=\frac{b+\sqrt{b^2-24ac}}{4c}\Big\},
\end{align*}
which correspond to the isotropic phase and nematic phase for liquid crystal materials respectively.

We are interested in the phase transition between the isotropic and the nematic phases.
Thus, it is assumed the critical phase transition regime, i.e.,
$F_b(0)=F_b(\QQ_*)$ for all $\QQ_*\in\mathcal{N}$, which is equivalent to the relation $b^2=27ac$. Then we have
\begin{align}\label{relation:abc}
a=\frac{cs_+^2}{3},\quad b=3cs_+,
\end{align}
and for uniaxial tensors $\QQ=s(\nn\nn-\frac13\II)$, it holds
\begin{align}\label{def}
F_b(\QQ)=\frac{c}{9}s^2(s-s_+)^2\triangleq \mathfrak{f}(s).
\end{align}

The main goal of this paper is to investigate the asymptotics of the energy functional \eqref{energy:LdG} as $\ve\to 0$ under suitable Dirichlet boundary data.

Note that if $L_2= 0$, the isotropic-nematic interface problem can be viewed as a special case of Keller-Rubinsterin-Sterberg problem \cite{RSK1,RSK2}.
The work \cite{LPW} by Lin-Pan-Wang can be applied with only slight modifications to obtain that the leading order energy of $ \mathcal{F}_\ve(\QQ,\nabla\QQ)$ is $\frac{1}{\ve}\sqrt{L_1}\mathcal{H}^2(\Gamma)$ while the next order is the Dirichlet energy
of the alignment vector field with homogeneous Neumann conditions on the interface. The regularity of the minimizers for limit problem is studied in \cite{LW19}.

The analysis of isotropic-nematic interface problem was initiated by de Gennes \cite{deGen1971}. In
\cite{deGen1971}, by simply considering the 1-D problem and assuming $\QQ$ is diagonal and uniaxial, de Gennes calculated the strength of surface tensions for both of normal and tangential conformations, which are given by
\begin{align}
c_0=c_*\beta,\quad \text{with }\beta= \left\{
  \begin{aligned}\label{value:tension-neg}
    \sqrt{L_1+\frac{2}{3}L_2},\quad & \text{for normal conformation};\\ 
  \sqrt{L_1+\frac{1}{6}L_2},\quad &\text{for tangential conformation}.
  \end{aligned}\right.
\end{align}
Here
\begin{align*}
c_*=\sqrt{\frac23}\int_{0}^{s_+}\sqrt{2\mathfrak{f}(s)} \ud s=\frac{s_+^3\sqrt{c}}{9\sqrt{3}},
\end{align*}
is a constant independent of $L_1, L_2$.
Then it is argued that near the isotropic-nematic interface, the normal conformation is favored for $L_2<0$ while the tangential conformation is favored for $L_2>0$. After then, there are a number of studies on the properties of order tensors near the interface \cite{DK, PSW, KBAM1, KBAM2, PS} by using both the numerics and asymptotic analysis. In these works, the equation of the static interface are solved on 1-D intervals or the whole real line under various anchoring conditions at the boundaries or infinities. It was shown that for homeotropic anchoring condition, there is an explicit uniaxial solution. In addition, this solution is proved to be stable for $L_2\le 0$ and unstable for $L_2>0$ \cite{PWZZ} in the 1-D setting.  However, for the tangential anchoring condition, the solution has to be biaxial unless $L_2=0$.

Recently, there has been a wide interests in the study of the high-dimension isotropic-nematic interface problem with anisotropic elasticity  based on the vectorial-type theories.
Golovaty et. al  \cite{GNSV1, GNSV2}  investigated the asymptotic limit of the isotropic-nematic transition problem with highly disparate elastic constants  in dimension two based on a $\mathbb{R}^2$-valued vector model. Liu \cite{Liu} considered the dynamic problem under the similar setting. In \cite{LW23}, Lin-Wang studied the isotropic-nematic interface problem based on the Ericksen model, which is a typical and important vectorial model. In particular, they obtained the asymptotic limit of the energy functional and rigorously justified that different  boundary conditions (homeotropic, trangential, or free boundary conditions) are intrinsically determined by different anisotropic elasticities under the limit process.

Up to the authors' knowledge, the rigorous analytical results on the 3-D isotropic-nematic interface based on the Landau-de Gennes $\QQ$-tensor theory are quite few. The vector models could be viewed as reduced Landau-de Gennes models under the assumption that the order tensors are uniaxial everywhere. This assumption are reasonable when one concerns the pure isotropic and nematic phases as the biaxiality could be neglected in these states. However, the biaxiality might not be neglected near the defects and the region of phase transion. Therefore, investigating the fine structure of particular configurations such as defects and phase transition interfaces within the Landau-de Gennes $\QQ$-tensor models would be of physical importance. 
We refer to \cite{FWZZ1} for formal analysis, and to \cite{FWZZ2, Laux-Liu} for the analysis of isotropic-nematic interfacial dynamics in the case $L_2=0$.

\subsection{Presentation of the main results}

Assume $\Omega$ is a bounded smooth domain in $\mathbb{R}^3$ with $\partial\Omega=\Sigma$. Let  $\Sigma^\pm$ be two disjoint connected open subsets of $\Sigma$ such that $\partial\Sigma^+=\partial\Sigma^-=\Sigma^0$ is a smooth closed curve of $\Sigma$, and $\Sigma=\Sigma^+\cup\Sigma^-\cup\Sigma^0$.
For given boundary data $\BB\in H^{1/2}(\partial\Omega)$, we define:
\begin{align*}
  H_{\BB}^1(\Omega)=\{\QQ\in H^1(\Omega),\quad \QQ|_{\partial\Omega}=\BB\}.
\end{align*}

We take the boundary data:
\begin{align}
\BB_\ve(x)=s_\ve(x)\EE_\ve(x), \quad \text{for }x\in \Sigma=\partial\Omega,
\end{align}
where $s_\ve$ is a scalar function and $\EE_\ve$ is $\mathcal{N}$-valued. Thus $\BB_\ve$ is uniaxial.

For the leading order asymptotic, we assume that:
\begin{itemize}
  \item[(A1)] There exists $M_0>0$ such that
  \begin{align}
&\|s_\ve\|_{L^2(\Sigma^-)}\to 0,\quad  \text{and } \|s_\ve-s_+\|_{L^2(\Sigma^+)}\to 0;\\
&\int_{\Sigma}\Big\{\ve\beta_L^2|\nabla_{\Sigma}s_\ve|^2+\frac{1}{\ve}\mathfrak{f}(s_\ve)\Big\}\ud \mathcal{H}^2\le M_0,\quad \text{for }\forall \ve\in(0,1);
\end{align}
\item[(A2)] There exists a map $\hat{\EE}_\ve\in H^1(\Omega, \mathcal{N})$ such that
\begin{align*}
\lim_{\ve\to 0}\ve\int_{\Omega}|\nabla\hat{\EE}_\ve|^2\ud x\to 0,\qquad \text{ with }\hat{\EE}_\ve|_{\Sigma}=\EE_\ve, ~\hat{\EE}_\ve|_{\Gamma}=\nu_\Gamma.
\end{align*}

\end{itemize}
Define
\begin{align*}
  F^*_{\ve}= \inf ~\big\{ \mathcal{F}_\ve(\QQ_\ve):~~\QQ_\ve\in H^1_{\BB_\ve}(\Omega) \big\}.
\end{align*}
The first result of our paper is the following characterization of the leading order of $ F^*_{\ve}$ as $\ve\to 0$ for $L_2<0$, which
gives a rigorous confirmation of de Gennes' formal derivation on the isotropic-nematic interface tension in the three dimensional space when $L_2<0$.
\begin{theorem}\label{thm:main1}
  Assume $-\frac32{L_1}<L_2<0$ and
\begin{itemize}
  \item $\Gamma\subset \Omega$ is an area-minimizing hypersurface with $\partial\Gamma=\Sigma$.
\item $\BB_\ve:\partial\Omega\to\mathcal{Q}$ satisfies the conditions.
\end{itemize}
  Then we have
  \begin{align*}
    \lim_{\ve\to 0} \ve F_\ve^* =\sqrt{{L_1}+\frac{2L_2}{3}}c_*\mathcal{H}^{n-1}(\Gamma).
  \end{align*}
\end{theorem}

For the next order asymptotics of the minimum $F_\ve^*$, we assume that  $\Gamma$ is the unique, strictly stable, area minimizing surface spanned by $\Sigma^0$. Let $d_{\Gamma}(x)$ be the signed distance function of $x$ to $\Gamma$.
Furthermore, we let
\begin{itemize}
\item $\EE=s_+\Big(\mathbf{b}\mathbf{b}-\frac13\II\Big)\in H^1(\partial\Omega, \mathcal{N})$ satisfies $\mathbf{b}\parallel \nu_{\Gamma}$ a.e. $x\in\Sigma^0$;
\item the function $S_R(z): [-R, R]\to [0,1]$ satisfies $S_R(-z)=1-S_R(z)$ for $|z|\le R$ and
\begin{align*}
 & S_R(R)=1,\quad  S_R(-R)=0,\\
 & \max_{|z|\le R}\big|\beta_L^2(S_R')^2+S_R^2(1-S_R)^2-2\beta_L\sqrt{S_R^2(1-S_R)^2}|S_R'|\big|\le Ce^{-2R/\beta_L},\\
  &\int_{-R}^R \big(\beta_L^2(S_R')^2+S_R^2(1-S_R)^2\big)\ud z\le \frac{\beta_L}3+Ce^{-R/\beta_L},\\
  & |S_R'(z)|\le C e^{-z/\beta_L},\text{ for }|z|\le R.
  \end{align*}
\end{itemize}
Then existence of $S_R$ for given $\beta_L>0$ is shown in Appendix.
Then we assume the boundary data
\begin{align*}
\BB_\ve(x)=s_\ve(x)\EE_\ve(x), \quad \text{for }x\in \partial\Omega,
\end{align*}
satisfies the properties:
\begin{itemize}
  \item[(A$1^*$)] $s_\ve\in H^1(\partial\Omega)$ satisfies:
\begin{align*}
 & s_\ve=S_{\ve^\gamma}(\frac{d_{\Gamma}(x)}{\ve}), \text{ for }x\in\partial\Omega\cap\{x:|d_{\Gamma}(x)|\le \ve^{\gamma}\};\\
 & \|s_\ve\|_{L^2(\Sigma^-\setminus\Gamma_{\ve^\gamma})}+\|s_\ve-s_+\|_{L^2(\Sigma^+\setminus\Gamma_{\ve^\gamma})}=O(\ve);\\
 & \lim_{\ve\to 0}\ve^\gamma\int_{\partial\Omega\setminus\Gamma_{\ve^\gamma}}|\nabla_{\Sigma}s_\ve|^2\ud\mathcal{H}^2=0.
\end{align*}
\item[(A$2^*$)] $\EE_\ve\to \EE$ in $H^1(\partial\Omega, \mathcal{N})$, and
 \begin{align*}
   \EE_\ve=\EE \text{ on } \partial\Omega\cap\Gamma_{\ve^{\gamma}}, \qquad \|\EE_\ve-\EE\|_{L^2(\partial\Omega\setminus\Gamma_{\ve^\gamma})}=o_{\ve}(1)\ve^\gamma.
 \end{align*}
\end{itemize}

\begin{theorem}\label{thm:main2}
  Assume $-\frac32{L_1}<L_2<0$ and
\begin{itemize}
  \item $\Gamma\subset \Omega$ is an area-minimizing hypersurface with $\partial\Gamma=\Sigma$.
\item $\BB_\ve=s_\ve\EE_\ve:\partial\Omega\to\mathcal{Q}$ satisfies the conditions.
\end{itemize}
  Then we have
  \begin{align*}
 F_\ve^* =\frac{1}{\ve}\sqrt{{L_1}+\frac{2L_2}{3}}c_*\mathcal{H}^{n-1}(\Gamma)+D_*+o_{\ve}(1).
  \end{align*}
Here
\begin{align*}
D_*= \inf_{\QQ^*\in \mathcal{A}}~~ \Big\{E(\QQ^*, \nabla\QQ^*)= \int_{\Omega^+} \Big[L_1|\nabla\QQ^*|^2+L_2|\nabla\cdot\QQ^*|^2\Big]\ud x\Big\}
\end{align*}
and
\begin{align*}
\mathcal{A}=\Big\{\QQ^*\in H^1(\Omega_+,\mathcal{N}), \QQ^*=\EE \text{ on }\Sigma^+, \QQ^*\barwedge\nu_\Gamma \text{ on }\Gamma\Big\}
\end{align*}
where $\nu_\Gamma$ is the unit normal of $\Gamma$.
\end{theorem}
\begin{Remark}
 Since $\QQ^*\in \mathcal{N}$, we can assume that $\QQ^*=s_+\Big(\nn\nn-\frac{\II}{3}\Big)$, then:
  \begin{align*}
 D_*=\inf_{\nn\in \tilde{\mathcal{A}}}~~ \bigg\{
   E(\nn, \nabla \nn)= \int_{\Omega^+} 2s_+^2\Big\{L_1|\nabla\nn|^2+L_2\big[(\nabla\cdot\nn)^2+|(\nn\cdot\nabla)\nn|\big]^2\Big\}\ud x\bigg\}
  \end{align*}
  with
  \begin{align*}
\tilde{\mathcal{A}}=\Big\{\nn\in H^1(\Omega_+,\mathbb{S}^2), \nn=\mathbf{b} \text{ on }\Sigma^+, \nn \parallel\nu_\Gamma \text{ on }\Gamma\Big\}.
\end{align*}
\end{Remark}

Comparing with the works \cite{GNSV2, LW23, Liu} on the asymptotic analysis for the vector models, some differences and new mathematical difficulties
will arise. Let us describe them in details.

First, the surface tension is independent of the anisotropic elasticity in the vector models, see \cite{LW23, Liu} for different settings.
While in the Landau-de Gennes model, as pointed by de Gennes \cite{deGen1971} and verified by various numerical simulations \cite{PSW, KBAM1,KBAM2}, the surface tension could be affected by the coefficient of anisotropic elasticity.
For $L_2<0$, a naive divergence-curl decomposition for $\QQ$-tensors gives us that
\begin{align*}
\frac{L_1+L_2}{2}|\nabla\QQ|^2+\frac{1}{\ve^2}F_b(\QQ)-\frac{L_2}{2}\sum_{i=1}^3|\mathrm{curl}\,\QQ_i|^2.
\end{align*}
Applying Lin-Pan-Wang's theory \cite{LPW}, the first two terms contribute a leading
order energy $$\frac{c_*}{\ve}\sqrt{L_1+L_2}\mathcal{H}^2(\Gamma).$$ However, the corresponding surface
tension $c_*\sqrt{L_1+L_2}$ is not consistent with the value $$c_*\sqrt{L_1+\frac{2L_2}{3}}$$  predicted by de Gennes \cite{deGen1971} and other numerical calculations.

Second, in vector models, such as Ericksen's model \cite{LW23} or  the modified vector model in \cite{GNSV2, Liu}, the order parameter mode and the direction mode can be decoupled straightforwardly, and the energy comprise terms like $|\nabla s_\ve \cdot \nn_\ve|^2$ ($|\nabla s_\ve \times \nn_\ve|^2$) or $|\nabla |\nn_\ve|\cdot\nn_\ve|^2$ ($|\nabla |\nn_\ve| \times \nn_\ve|^2$). Formally, one can derive that these terms contribute interfacial energy $\frac{1}{\ve}|\nn\cdot\nu_{\Gamma}|^2\ud \mathcal{H}^2(\Gamma)$ ($\frac{1}{\ve}|\nn\times\nu_{\Gamma}|^2\ud \mathcal{H}^2(\Gamma)$), minimization of which gives the tangential or homeotropic boundary conditions on the interface (although the rigorous confirmations are nontrivial). However, at least at a first glance, such a mechanism for the Landau-de Gennes models is not obvious and it is not clear why the homeotropic anchoring condition is favored for negative $L_2$.

The main novelty of this paper are two folds. First, we introduce a div-curl decomposition for symmetric trace-free tensors (see Lemma \ref{lem:decomp})
\begin{align}\label{intro:div-curl}
|\nabla\QQ|^2=\frac32|\nabla\cdot\QQ|^2+\frac{1}{4}|\epsilon^{ikl}\partial_kQ_{lj}+\epsilon^{jkl}\partial_kQ_{li}|^2+ (\text{divergence terms}). 
\end{align}
Using \eqref{intro:div-curl}, we can write that
\begin{align*}
 \mathcal{F}_\ve(\QQ,\nabla\QQ)= &\int_{\Omega}\Big\{ \frac12\Big({L_1}+\frac{2L_2}{3}\Big)|\nabla\QQ|^2+\frac{1}{\ve^2}F_b(\QQ) \\
 &\quad -\frac{L_2}{12}|\epsilon^{ikl}\partial_kQ_{lj}+\epsilon^{jkl}\partial_kQ_{li}|^2+ (\text{divergence terms})~~ \Big\}\ud x.
\end{align*}
Then applying Lin-Pan-Wang's theory \cite{LPW}, the first two terms give a leading
order energy $$\frac{1}{\ve}\sqrt{L_1+\frac{2L_2}{3}}c_*\mathcal{H}^2(\Gamma),$$
which is consistent with the surface tension predicted by de Gennes \cite{deGen1971} for $L_2<0$. We also believe that this decomposition is of independent interest for the analysis of Landau–de Gennes Q-tensor theory. For example, in a companion paper by Dong-Ren-Wang \cite{DRW}, this decomposition plays a key role in the asymptotic analysis of the isotropic–nematic transition dynamics.

Second, even with help of the decomposition, it is not clear how the curl term could give the heometropic boundary condition on the interface.
To apply the density arguments as in \cite{LW23}, one would need the relation between $\QQ$-tensors and its quasi-distance $\psi(\QQ)$, which could be rather complicated for  biaxial tensors. We use the strong convergence of order tensors to bypass this difficulty and then obtain the boundary condition directly.

The rest of the paper is organized as follows: In Section 2, we present some analysis on the bulk energy  $F_b(\QQ)$. More importantly, a novel decomposition for $|\nabla \QQ|^2$ is introduced. Sections 3 and 4 are devoted to proving the asymptotic limit of the leading order and the next order respectively.
\medskip

\section{Preliminary}
\subsection{Analysis on the bulk energy $F_b(\QQ)$}

For $\QQ\in \mathcal{Q}$, we define the quasi-distance between $0$ and $Q$:
\begin{align*}
  d_0^F(\QQ)=\inf_{\substack{P(t)\in C^1(\mathbb{R}, \mathcal{Q})\\ P(-\infty)=0, P(+\infty)=Q}} \int_{\mathbb{R}}\Big\{ \frac{1}{2}|P'(t)|^2+F_b(P(t))\Big\} \ud t.
\end{align*}
In particular, we define the constant
\begin{align}
  c_*:= d_0^F(\QQ),\qquad \text{for } \QQ=s_+\Big(n\otimes n-\frac13 I\Big)\in \mathcal{N},
\end{align}
which is independent of the choice of $n$. Indeed, in this case, the infimum is attained by $P(t)=s(t)\Big(n\otimes n-\frac13 I\Big)$ where
$s(-\infty)=0, s(+\infty)=s_+$ and $s(t)$ minimizes the energy
\begin{align*}
  \int_{\mathbb{R}}\Big\{\frac13(s')^2+\frac{c}{9}s^2(s-s_+)^2 \Big\}\ud t.
\end{align*}
More precisely, we have
\begin{align*}
  s(t)=\frac{s_+}{1+\exp(-s_+\sqrt{\frac{c}{3}}t)},\text{ and }c_*=\frac{s_+^3\sqrt{c}}{9\sqrt{3}}.
\end{align*}

\begin{Lemma}There exist constants $C_1, C_2>0$ such that
\begin{align*}
C_1(|\QQ|^2+|\QQ|^3)\le d_0^F(\QQ)\le C_2 (|\QQ|^2+|\QQ|^3).
\end{align*}
\end{Lemma}
\begin{proof}
One can directly get that for any $\QQ\in\mathcal{Q}$:
\begin{align*}
   F_b(\QQ)\le C(|\QQ|+|\QQ|^2)^2.
\end{align*}
Thus, take $\bar{\gamma}(s)=\frac{s}{L}\QQ$ for $s\in[0,L]$ with $L=|\QQ|$. Then $|\bar{\gamma}'(s)|=1$ and
\begin{align*}
d_0^F(\QQ)\le  \int_{0}^L\sqrt{2F_b(\bar{\gamma}(s))}\ud s\le  C\int_{0}^L(|\bar{\gamma}(s)|+|\bar{\gamma}(s)|^2)\ud s \\ =C\int_{0}^L(|\bar{\gamma}(s)|+|\bar{\gamma}(s)|^2)|\gamma'(s)|\ud s \le C (|\QQ|^2+|\QQ|^3).
\end{align*}

On the other hand, there exist $\delta, R>0$ and $c_1>0$ with $\delta $ small and $R$ large, such that
\begin{align*}
 F_b(\QQ)\ge c_1|\QQ|^2,\text{ for }|\QQ|\le \delta;\quad F_b(\QQ)\ge c_1|\QQ|^4,\text{ for }|\QQ|\ge R;
\end{align*}
For any $\gamma:[0,L]\to \mathcal{Q}$ with $\gamma(0)=0,\gamma(L)=\QQ, |\gamma'(s)|=1$, one has $|\QQ|\le \delta$ that
\begin{align*}
  \int_{0}^L\sqrt{2F_b(\gamma(s))}\ud s\ge  c_1\int_{0}^L\sqrt{2}|\gamma(s)|\ud s =c_1\int_{0}^L\sqrt{2}|\gamma(s)||\gamma'(s)|\ud s \ge C_1|\QQ|^2,
\end{align*}
which yields that $d_0^F(\QQ)\ge C_1|\QQ|^2$.

For $|\QQ|\ge R$, we choose $s_0\in(0,L)$ such that $|\QQ(s_0)|=R$ and $|\QQ(s)|\ge R$  for $s\ge s_0$.
Then we have
\begin{align*}
  \int_{0}^L\sqrt{2F_b(\gamma(s))}\ud s\ge& C_1\delta^2+ c_1\int_{s_0}^L\sqrt{2}|\gamma(s)|^2\ud s \\ =&C_1\delta^2+c_1\int_{0}^L\sqrt{2}|\gamma(s)|^2|\gamma'(s)|\ud s
  \ge C(|\QQ|^3-R^3+1)
\end{align*}
which gives $d_0^F(\QQ)\ge C_1|\QQ|^3$ for $|\QQ|>R$.
\end{proof}

For $\QQ\in\mathcal{Q}$, we define
\begin{align}
 d_\mathcal{N}(\QQ) :=d(\QQ, \mathcal{N})= \inf_{\PP\in \mathcal{N}} |\PP-\QQ|,
\end{align}
and
\begin{align*}
  f_*(\rho)=\frac{a \rho^2}{2}-\frac{b\rho^3}{3\sqrt{6}}+\frac{c\rho^4}{4}=\frac{c}{4}\rho^2
  \Big(\rho-\sqrt{\frac{2}{3}}s_+\Big)^2=\mathfrak{f}(\sqrt{\frac{2}{3}}\rho).
\end{align*}
Then we have the following conclusions.
\begin{Lemma}\label{lem:bulk-ineq}
(1): $F_b(\QQ)\ge f_*(|\QQ|), \text{ for }\forall \QQ\in\mathcal{Q}$;

(2): If $d_\mathcal{N}(\QQ)\in \Big[0,{\sqrt{\frac23}}s_+\Big]$,  then $F_b(\QQ)\ge f_*(d_\mathcal{N}(\QQ))$;

(3): There exists $C>0$ such that $F_b(\QQ)\ge C |\QQ|^2d_\mathcal{N}^2(\QQ)$, for $\forall \QQ\in\mathcal{Q}$. In particular, we have
\begin{align*}
  F_b(\QQ)\ge C d^2(\QQ, \mathcal{N}\cup \{0\}).
\end{align*}
\end{Lemma}

\begin{proof}
The first inequality is a consequence of the fact that
\begin{align*}
  \mathrm{tr}\, \QQ^3\ge -\frac{|\QQ|^3}{\sqrt{6}}.
\end{align*}
For the second one, it suffices to consider the diagonal case. Assume $\QQ=\text{diag}~\{\lambda_1, \lambda_2, \lambda_3\}$ with $\lambda_1+\lambda_2+\lambda_3=0$ and $\lambda_1\ge \lambda_2\ge \lambda_3$. Let $r=d_\mathcal{N}(\QQ)$. Then  $$r^2=(\lambda_1-\frac{2}{3}s_+)^2+(\lambda_2+\frac{1}{3}s_+)^2+(\lambda_3+\frac{1}{3}s_+)^2.$$
Set
\begin{align*}
  s_1=\lambda_1-\frac{2}{3}s_+,\quad s_2=\lambda_2+\frac{1}{3}s_+,\quad s_3=\lambda_3+\frac{1}{3}s_+.
\end{align*}
Then $2s_2s_3=(s_2+s_3)^2-s_2^2-s_3^2=2s_1^2-r^2$ and
\begin{align*}
\sum_{i=1}^3\lambda_i^2=r^2+\frac{2}{3}s_+^2+2s_1s_+,\quad
\sum_{i=1}^3\lambda_i^3=3\lambda_1\lambda_2\lambda_3=3(s_1+\frac{2}{3}s_+)(s_1^2-\frac12 r^2+\frac{1}{3}s_1s_++\frac{1}{9}s_+^2),
\end{align*}
which gives
\begin{align*}
F_b(\QQ)=&\frac{a}{2}\sum_{i=1}^3\lambda_i^2-\frac{b}{3}\sum_{i=1}^3\lambda_i^3+\frac{c}{4}(\sum_{i=1}^3\lambda_i^3)^2\\
=&\frac{a}{2}(r^2+\frac{2s_+^2+6s_1s_+}{3})-{b}(s_1+\frac{2}{3}s_+)(s_1^2-\frac12 r^2+\frac{3s_1s_++s_+^2}{9})+\frac{c}{4}(r^2+\frac{2s_+^2+6s_1s_+}{3})^2.
\end{align*}
Hence, by setting $t=3s_1s_++s_+^2$, we obtain
\begin{align*}
 F_b(\QQ)-f(r)=& \frac{at}{3}-{b}(s_1+\frac{2}{3}s_+)(s_1^2-\frac12 r^2+\frac{t}{9})+\frac{ct}{3}(r^2+\frac{t}{3})+\frac{br^3}{3\sqrt{6}}\\
 =&\frac{c}{9}\Big(ts_+^2-(t+s_+^2)(\frac{(t-s_+^2)^2}{s_+^2}-\frac92 r^2+t)+t(3r^2+t)+\frac{9s_+r^3}{\sqrt{6}}\Big)\\
 =&\frac{c}{9}\Big(-(t+s_+^2)(\frac{(t-s_+^2)^2}{s_+^2}-\frac92 r^2)+t3r^2+\frac{9s_+r^3}{\sqrt{6}}\Big)\\
 =&\frac{c}{9}\Big((t+s_+^2)(6r^2-\frac{(t-s_+^2)^2}{s_+^2})+\frac{3r^2}{2}(t-s_+^2+\sqrt{6}rs_+)\Big).
\end{align*}
As $r^2=s_1^2+s_2^2+s_3^2\ge \frac{3}{2}s_1^2$, we have $s_1\in[-\sqrt{\frac{2}{3}}r, \sqrt{\frac{2}{3}}r]$ and then
\begin{align*}
  -\sqrt{6}rs_++s_+^2\le t\le  \sqrt{6}rs_++s_+^2,
\end{align*}
which implies
\begin{align*}
6r^2-\frac{(t-s_+^2)^2}{s_+^2}\ge 0, \quad \text{or equivalently, } t-s_+^2+\sqrt{6}rs_+\ge 0.
\end{align*}
On the other hand, as $r\le \sqrt{\frac{2}{3}}s_+$, we have
\begin{align*}
  t+s_+^2\ge s_+(2s_+-\sqrt{6}r)\ge 0.
\end{align*}
All above together gives that $F_b(\QQ)-f_*(r)\ge 0$, which concludes the proof.
\end{proof}

\subsection{A novel decomposition for $|\nabla\QQ|^2$}
For $\QQ\in\mathcal{Q}$, we define $\mathcal{T}(\QQ)\in \mathcal{Q}$ as:
\begin{align}\label{def:TQ}
\mathcal{T}(\QQ) = \big(T_{ij}(\QQ)\big)_{1\le i, j\le 3},\quad \text{with } T_{ij}(\QQ)=\epsilon^{ikl}\partial_kQ_{lj}+\epsilon^{jkl}\partial_kQ_{li}.
\end{align}

\begin{Lemma}[Div-curl decomposition for gradient of $\QQ$-tensor field] \label{lem:decomp}
  For symmetric trace-free tensor $\QQ$, it holds:
  \begin{align}\label{eq:prop:div-grad}
|\nabla\QQ|^2=\frac{3}{2}|\nabla\cdot\QQ|^2+\frac{1}{4} |\mathcal{T}(\QQ)|^2+(\partial_kQ_{li}\partial_lQ_{ki}-\partial_kQ_{ki}\partial_lQ_{li}).
  \end{align}
\end{Lemma}
\begin{Remark}\label{rmk:bound-integral}
Let $\Omega\subset\mathbb{R}^3$ be a bounded domain and $\nu$ be the out normal of $\partial\Omega$.
One can straightforwardly obtain
  \begin{align*}
 \int_{\Omega} (\partial_if\partial_j g- \partial_jf\partial_i g)\ud x=\int_{\Omega}\Big(\partial_i(f\partial_jg)-\partial_j(f\partial_ig)\Big) \ud x
 =\int_{\partial\Omega}f(\nu_i\partial_jg-\nu_j\partial_i g)\ud\sigma.
  \end{align*}
As $\nu_i\partial_jg-\nu_j\partial_i g$ are components of $\nu\times \nabla g$ which only depends on the boundary value of $g$, the last term in \eqref{eq:prop:div-grad} depends only on the boundary data of $\QQ$. Thus, the identity \eqref{eq:prop:div-grad} can be viewed as an analogous version of the div-curl decomposition for the gradients of vector fields:
\begin{align*}
  |\nabla \nn|^2=|\nabla\cdot\nn|^2+|\nabla\times\nn|^2+(\partial_in_j\partial_jn_i-\partial_in_i\partial_jn_i).
\end{align*}
\end{Remark}
\begin{proof}It is straightforward to obtain that
\begin{align*}
|\mathcal{T}(\QQ)|^2&=|\ve^{ikl}\partial_k\QQ_{lj}+\ve^{jkl}\partial_k\QQ_{li}|^2\\
&=2\ve^{ikl}\partial_k\QQ_{lj}\ve^{imn}\partial_m\QQ_{nj}+2\ve^{ikl}\partial_k\QQ_{lj}\ve^{jmn}\partial_m\QQ_{ni}.
\end{align*}
Since
\begin{align*}
  \ve^{ikl}\partial_k\QQ_{lj}\ve^{imn}\partial_m\QQ_{nj}=|\nabla\QQ|^2-\partial_kQ_{li}\partial_lQ_{ki}
\end{align*}
and
\begin{align*}
  \ve^{ikl}\partial_k\QQ_{lj}\ve^{jmn}\partial_m\QQ_{ni}&\,=\left(\begin{array}{ccc}
\delta^i_j&\delta^i_m&\delta^i_n\\
\delta^k_j&\delta^k_m&\delta^k_n\\
\delta^l_j&\delta^l_m&\delta^l_n\\
		\end{array}\right) \partial_k\QQ_{lj}\partial_m\QQ_{ni}\\
&\,=\partial_kQ_{li}\partial_kQ_{li}+\partial_kQ_{ll}\partial_iQ_{ki}+\partial_kQ_{lk}\partial_lQ_{ii}\\
& \,\quad -\partial_kQ_{ll}\partial_kQ_{ii}-\partial_kQ_{li}\partial_lQ_{ki}-\partial_kQ_{lk}\partial_iQ_{li}\\
&\,=\partial_kQ_{li}\partial_kQ_{li}-\partial_kQ_{li}\partial_lQ_{ki}-\partial_kQ_{lk}\partial_iQ_{li}\\
&\,=|\nabla\QQ|^2-\partial_kQ_{li}\partial_lQ_{ki}-|\nabla\cdot\QQ|^2,
\end{align*}
we have
\begin{align*}
|\mathcal{T}(\QQ)|^2&=2(2|\nabla\QQ|^2-2\partial_kQ_{li}\partial_lQ_{ki}-|\nabla\cdot\QQ|^2)\\
&=2(2|\nabla\QQ|^2-3|\nabla\cdot\QQ|^2)+4(\partial_kQ_{lk}\partial_iQ_{li}-\partial_kQ_{li}\partial_lQ_{ki}).
\end{align*}
which gives the conclusion.
\end{proof}

As the discussion in Remark \ref{rmk:bound-integral}, we have
\begin{align*}
\int_{\partial\Omega}(\partial_jQ_{ik}\partial_kQ_{ij}-\partial_jQ_{ij}\partial_kQ_{ik})
&=\int_{\partial\Omega} Q_{ik}(\nu_j\partial_kQ_{ij}-\nu_k\partial_jQ_{ij})\\&=
\int_{\partial\Omega} Q_{ik}\ve^{ljk}(\nu\times \nabla)_l Q_{ij}.
\end{align*}
We introduce
\begin{align}\label{operator:boundary}
  \mathcal{B}_{\partial\Omega}(\QQ):=\frac{2}{3}\int_{\partial\Omega} Q_{ik}\ve^{ljk}(\nu\times \nabla)_l Q_{ij}\ud\sigma.
\end{align}
Note that $(\nu\times \nabla)_l$ are tangential derivative operators on $\partial\Omega$. Therefore $\mathcal{B}_{\partial\Omega}(\QQ)$
depends only on the boundary value of $\QQ$ on $\partial\Omega$, and apparently,
\begin{align*}
    |\mathcal{B}_{\partial\Omega}(\QQ)|\le C \|\QQ\|^2_{H^{1}(\partial\Omega)}.
\end{align*}
A standard extension argument gives us the improved estimate:
\begin{align*}
  |\mathcal{B}_{\partial\Omega}(\QQ)|\le C \|\QQ\|^2_{H^{\frac{1}{2}}(\partial\Omega)}.
\end{align*}

Define
\begin{align}\label{def:barwedgeQQ}
  \QQ\barwedge \Phi &:=(\epsilon^{kij} Q_{im}\Phi_{jm}+\epsilon^{kij}Q_{im}\Phi_{jm})_{1\le k\le 3}\in \mathbb{R}^3 &&\text{ for }\QQ, \Phi\in\mathcal{Q},\\ \label{def:barwedgeQv}
  \QQ\barwedge \nu&:=\Big((\epsilon^{ikl} Q_{kj}+\epsilon^{jkl}Q_{ki})\nu_l\Big)_{1\le i, j\le 3}\in\mathcal{Q} &&\text{ for }\QQ\in\mathcal{Q}, \nu\in\mathbb{R}^3.
\end{align}
Direct calculations give us that for $\QQ, \Phi\in\mathcal{Q}$ and $\nu\in\mathbb{R}^3$:
\begin{align*}
  \Phi:(\QQ\barwedge\nu)=-\QQ:(\Phi\barwedge\nu) =-\nu\cdot(\QQ\barwedge\Phi).
\end{align*}
Moreover, for $\QQ=s_+(\nn\nn-\frac13\II)\in\mathcal{N}$ and $\nu\in\mathbb{R}^3$, one has
\begin{align*}
   \QQ\barwedge \nu&=s_+\Big((\nn\wedge \nu)\nn+\nn(\nn\wedge \nu)\Big),
\end{align*}
and in this case, $ \QQ\barwedge \nu=0$ if and only if $\QQ=s_+(\nu\nu-\frac13\II)$. Here $\nn\wedge \nu$ is the standard wedge for two vector in $\mathbb{R}^3$.

Direct calculation gives us the Stokes-type formula for the operator $\mathcal{T}$:
\begin{align}\label{eq:wedge-integral}
\int_{U}\mathcal{T}(\QQ):\Phi \ud x&=\int_{U} (\epsilon^{ikl}\partial_kQ_{lj}+\epsilon^{jkl}\partial_kQ_{li}):\Phi_{ij}\ud x\nonumber\\
&=-\int_{U} (\epsilon^{ikl}Q_{lj}+\epsilon^{jkl}Q_{li}):\partial_k\Phi_{ij}\ud x+\int_{\partial U}\nu_k(\epsilon^{ikl} Q_{lj}+\epsilon^{jkl}Q_{li})\Phi_{ij}\ud \mathcal{H}^2 \nonumber\\
&= \int_{U}\QQ:\mathcal{T}(\Phi) \ud x-\int_{\partial U}\Phi:(\QQ\barwedge\nu)\ud \mathcal{H}^2.
\end{align}

\bigskip

\section{Asymptotics of the leading order}
Define
\begin{align*}
  \mathcal{F}_\ve(\QQ)=&\int_{\Omega}\Big[\frac{1}2(L_1|\nabla \QQ_\ve|^2+{L}_2|\nabla\cdot\QQ_\ve|^2)+\frac{1}{\ve^2}F_b(\QQ_\ve)\Big]\ud x,\\
 \widetilde{ \mathcal{F}}_\ve(\QQ)=&\int_\Omega  \Big[\frac12\Big(L_1+\frac{2{L}_2}{3}\Big)|\nabla \QQ_\ve|^2-\frac{{L}_2}{12}|\mathcal{T}(\QQ_\ve)|^2+\frac{1}{\ve^2}F_b(\QQ_\ve)\Big]\ud x.
\end{align*}
Then by the decomposition formula \eqref{eq:prop:div-grad}, one has
\begin{align*}
\mathcal{F}_\ve(\QQ) =  \widetilde{ \mathcal{F}}_\ve(\QQ)  +L_2 \mathcal{B}_{\partial\Omega}(\BB_{\ve}).
\end{align*}
Thus, for $L_2\in(-\frac{3}{2}L_1,0)$, one may apply Lin-Pan-Wang's theory \cite{LPW} to obtain the leading order asymptotics.
The main difference is that the assumption in \cite{LPW} that the bulk energy $F(u)$
the proportional to the square of distance between

Define
\begin{align*}
  F_\ve^*=\inf_{\QQ\in H_{\BB_\ve}^1}\mathcal{F}_\ve(\QQ),\quad  \widetilde{F}_\ve^*=\inf_{\QQ\in H_{\BB_\ve}^1}\widetilde{\mathcal{F}}_\ve(\QQ),
\end{align*}
Then we have
\begin{align*}
  F_\ve^*=  \widetilde{F}_\ve^*+L_2\mathcal{B}_{\partial\Omega}(\BB_{\ve}).
\end{align*}
For ${L_2}>-\frac{3}{2}L_1$, let:
\begin{align*}
  \beta_L=\sqrt{L_1+\frac{2{L}_2}{3}}.
\end{align*}

\subsection{Lower bound estimate}

Assume that $\{\QQ_\ve\}\subset H_{\BB_\ve}^1(\Omega)$ and there holds the energy bound
\begin{align*}
  \ve \mathcal{F}_\ve(\QQ_\ve )\le \Lambda
\end{align*}
for some uniform positive constant $\Lambda$. The proof follows the steps in \cite[Proof of Theorem 1.1]{LPW}.

\begin{Proposition} It holds that
\begin{align}\label{ineq:energy-lead}
\ve\int_{\Omega}\Big[\frac12(L_1|\nabla \QQ_\ve|^2+L_2|\nabla\cdot\QQ_\ve|^2)+\frac{1}{\ve^2}F_b(\QQ_\ve)\Big]\ud x\ge {\beta_L}c_*\mathcal{H}^{n-1}(\Gamma)+o_{\ve}(1).
\end{align}
\end{Proposition}
\begin{proof}
We assume that
\begin{align*}
 \Lambda_1: =\liminf_{\ve\to 0} \ve \mathcal{F}_\ve(\QQ_\ve) <+\infty.
\end{align*}

For any $\delta\in(0,\frac{1}2d(\mathcal{N},0))$, define
\begin{align*}
  \Omega_{\ve,\delta}^+=\{x\in\Omega: d(\QQ_\ve(x), \mathcal{N})<\delta\}, \quad
  \Omega_{\ve,\delta}^-=\{x\in\Omega: d(\QQ_\ve(x), 0)<\delta\},
\end{align*}
and
\begin{align*}
  E_{\ve,\delta}=\Omega\setminus(  \Omega_{\ve,\delta}^+\cup  \Omega_{\ve,\delta}^-).
\end{align*}
Lemma \ref{lem:bulk-ineq} gives us that there exists $C>0$ such that
\begin{align*}
  F_b(\QQ)\ge C\delta^2,\quad \text{ for }d(\QQ_\ve, \mathcal{N}\cup\{0\})\ge \delta.
\end{align*}
Thus, we have
\begin{align}\label{estimate:inner-domain}
  |E_{\ve,\delta}|\le \frac{1}{C\delta^2}\int_{\Omega}F_b(\QQ_\ve) \ud x\le \frac{\Lambda \ve}{C{\delta}^2}\to 0,\qquad \text{ as }\ve\to 0.
\end{align}

Moreover,
\begin{align*}
 \Lambda &~\ge \ve \mathcal{F}_\ve(\QQ_\ve )\\
 &~ =\ve \widetilde{ \mathcal{F}}_\ve(\QQ_\ve)  +\ve  L_2\mathcal{B}_{\partial\Omega}(\BB_{\ve})\\
&~\ge  \int_\Omega\Big\{ \frac{\ve\beta_L^2}{2}|\nabla\QQ_\ve|^2+\frac{1}{\ve}F_b(\QQ_\ve)\Big\}\ud x -C\ve \\
 &~ \ge \beta_L\int_\Omega \sqrt{2F_b(\QQ_\ve)}|\nabla\QQ_\ve|\ud x -C\ve .
\end{align*}
Therefore,
\begin{align*}
 \Lambda +1~&\ge \beta_L\int_\Omega \sqrt{2F_b(\QQ_\ve)}|\nabla\QQ_\ve|\ud x \\
& \ge  \beta_L\Big(  \int_{\{x\in\Omega: d(\QQ_\ve, \mathcal{N})\le \delta \}} + \int_{\{x\in\Omega: d(\QQ_\ve, 0)\le \delta \} } \Big)\sqrt{2F_b(\QQ_\ve)}|\nabla\QQ_\ve|\ud x\\
& \ge C   \int_{\{x\in\Omega: d(\QQ_\ve, \mathcal{N})\le \delta \}}d(\QQ_\ve, \mathcal{N}) |\nabla\QQ_\ve|\ud x+ \int_{\{x\in\Omega: d(\QQ_\ve, 0)\le \delta \}} d(\QQ_\ve, 0) |\nabla\QQ_\ve|\ud x\\
& \ge C   \int_{\{x\in\Omega: d(\QQ_\ve, \mathcal{N})\in[\delta/2,\delta] \}}d(\QQ_\ve, \mathcal{N}) |\nabla\QQ_\ve|\ud x+ \int_{\{x\in\Omega: d(\QQ_\ve, 0)\in[\delta/2,\delta] \}} d(\QQ_\ve, 0) |\nabla\QQ_\ve|\ud x\\
&\ge C\delta \Big(\int_{\{x\in\Omega: d(\QQ_\ve, \mathcal{N})\in[\delta/2,\delta] \}}|\nabla d(\QQ_\ve, \mathcal{N})|\ud x
+\int_{\{x\in\Omega: d(\QQ_\ve, 0)\in[\delta/2,\delta] \}} |\nabla d(\QQ_\ve, 0)|\ud x\Big)\\
&\ge C\delta \int_{\delta/2}^\delta \mathcal{H}^2\Big(\{x\in\Omega: d(\QQ_\ve, \mathcal{N})=t\}\cup\{x\in\Omega: d(\QQ_\ve, 0)=t\}\Big)\ud t.
\end{align*}
Thus, there exists $\delta_{\ve}\in[\delta/2,\delta]$ such that
\begin{align*}
  \mathcal{H}^2\Big(\{x\in\Omega: d(\QQ_\ve, \mathcal{N})=\delta_{\ve}\}\cup\{x\in\Omega: d(\QQ_\ve, 0)=\delta_{\ve}\}\Big)\le C,
\end{align*}
which implies
\begin{align}\label{estimate:boundary}
  \mathcal{H}^2(\partial \Omega_{\ve,\delta_{\ve}}^+\cap \Omega )+  \mathcal{H}^2(\partial \Omega_{\ve,\delta_{\ve}}^-\cap \Omega )\le C.
\end{align}

From \eqref{estimate:inner-domain}-\eqref{estimate:boundary}, there are two sets
$E^\pm\subset\Omega$ such that, up to choice of subsequences,
\begin{align*}
\chi_{\Omega_{\ve,\delta_{\ve}}^\pm}\rightharpoonup \chi_{E^\pm}\quad \text{in } BV(\mathbb{R}^3),\\
\chi_{\Omega_{\ve,\delta_{\ve}}^\pm}\rightarrow \chi_{E^\pm}\quad \text{in } L^1(\mathbb{R}^3),
\end{align*}
as $\ve\to 0$. Moreover, one get from (\ref{estimate:inner-domain}) that
\begin{align*}
  |E^+\cap E^-|=0, \quad |\Omega\setminus(E^+\cup E^-)|=0.
\end{align*}

Define
\begin{align*}
  \phi_\ve(x)=d_0^F(\QQ_\ve(x)).
\end{align*}
As $|D_\QQ d_0^F|=\sqrt{2F_b(\QQ)}$, we have
\begin{align*}
  |\nabla\phi_\ve |\le \sqrt{2F_b(\QQ_\ve(x))}|\nabla\QQ_\ve|,
\end{align*}
which implies that
\begin{align*}
 \beta_{L} \int_{\Omega}|\nabla\phi_\ve|\ud x\le \int_{\Omega}\Big(\frac{\beta_{L}^2\ve}{2}|\nabla\QQ_\ve|^2+\frac{F_b(\QQ_\ve)}{\ve}\Big)\ud x\le \Lambda.
\end{align*}
Since $\phi_\ve$ are uniformly bounded on $\partial\Omega$, we get from the Poincare inequality that
\begin{align*}
  \int_\Omega|\phi_\ve|\ud x\le C(1+\int_\Omega|\nabla\phi_\ve|\ud x)\le C(\Lambda+1).
\end{align*}
Thus $\{\phi_\ve\}$ are uniformly bounded in $BV(\Omega)$, and there is a nonnegative function $\phi$ in $BV(\Omega)$ such that,
 up to choice of subsequences, $\phi_\ve\to \phi$ weakly in $BV(\Omega)$ and strongly in $L^1(\Omega)$.
We get by the lower semicontinuity that
\begin{align}\label{DPhi-bound}
\beta_{L}|D\phi|(\Omega)\le \beta_{L}\liminf_{\ve\to 0} \int_{\Omega}|\nabla\phi_\ve|\ud x\le \liminf_{\ve\to 0}\ve\mathcal{F}(\QQ_\ve)=\Lambda_1.
\end{align}

Using Fatou's lemma, we can get
\begin{align*}
  \int_{{\Omega}} \phi^2\chi_{E_-} \ud x
 & \le \liminf_{\ve\to 0}   \int_{{\Omega}} \phi_\ve^2\chi_{\Omega^-_{\ve,\delta_{\ve}}} \ud x
  \le C \liminf_{\ve\to 0}   \int_{{\Omega}} |\QQ_\ve|^2\chi_{\Omega^-_{\ve,\delta_{\ve}}} \ud x\\
&  \le C \liminf_{\ve\to 0}   \int_{{\Omega}} F_b(\QQ_\ve)\chi_{\Omega^-_{\ve,\delta_{\ve}}} \ud x =0,
\end{align*}
and
\begin{align*}
  \int_{{\Omega}} |\phi-c_*|^2\chi_{E_+} \ud x\le &~ \liminf_{\ve\to 0}   \int_{{\Omega}} (\phi_\ve-c_*)^2\chi_{\Omega^+_{\ve,\delta_{\ve}}} \ud x\\
  =&~\liminf_{\ve\to 0}   \int_{\Omega^+_{\ve,\delta_{\ve}}}(d_0^F(\QQ_\ve)-d_0^F(\Pi\QQ_\ve))^2 \ud x\\
  \le &~C \liminf_{\ve\to 0} \int_{\Omega^+_{\ve,\delta_{\ve}}}|\QQ_\ve-\Pi\QQ_\ve|^2\ud x\\
  \le &~C\liminf_{\ve\to 0} \int_{\Omega^+_{\ve,\delta_{\ve}}}F_b(\QQ_\ve)\ud x =0.
 \end{align*}
Here $\Pi\QQ_\ve$ denotes the  nearest-point projection of $\QQ_\ve$ to the uniaxial manifold $\mathcal{N}$.
 Thus $\phi=0$ for a.e. $x\in E_-$ and $\phi=c_*$ for a.e. $x\in E_+$.

As a consequence, one can apply the arguments in \cite[Proposition 3.1]{LPW} to obtain \eqref{ineq:energy-lead}. More precisely,
from \eqref{estimate:boundary}, one has
 \begin{align}\label{Eplus}
   E^+=\{x\in\Omega: \quad \phi\ge t\},\quad \forall t\in[0,c_*).
 \end{align}
For a subset $E\subset\mathbb{R}^3$, we denote by $\llbracket E\rrbracket$ the corresponding 3-dimensional current, and denote by $\partial \llbracket E\rrbracket$ the boundary current of $\llbracket E\rrbracket$.
Then we have
\begin{align*}
  \partial \llbracket \partial\llbracket\Omega^+_{\ve,\delta_{\ve}}\rrbracket\lfloor \Omega \rrbracket=\{x\in\partial\Omega, d(\QQ_\ve, \mathcal{N})=\delta_{\ve}\}\rightharpoonup \llbracket\Sigma^0\rrbracket \text{ as }\ve\to 0,
\end{align*}
as weak convergence of currents. Thus, we have
\begin{align}\label{E-boundary}
   \partial \llbracket \partial\llbracket E^+\rrbracket\lfloor \Omega \rrbracket =\Sigma^0.
\end{align}

Combining \eqref{Eplus}-\eqref{E-boundary} and the minimality of $\Gamma$ that
\begin{align*}
  \mathcal{H}^2(\partial^*\{x\in\Omega:\phi(x)\ge t\}\lfloor \Omega ) \ge  \mathcal{H}^2(\Gamma), \qquad \forall t\in[0,c_*),
\end{align*}
where $\partial^*E$ denotes the reduce boundary of $E$.
Thus, we have by \eqref{DPhi-bound} and the co-area formula for BV functions that
\begin{align*}
  \Lambda_1 \ge \beta_L |D\phi|(\Omega)=\beta_L \int_{\mathbb{R}_+}\mathcal{H}^2\big(\partial^*\{x\in\Omega:\phi>t\}\lfloor \Omega\big) dt\\
 \ge \beta_L \int_0^{c_*} \mathcal{H}^2\big(\partial^*\{x\in\Omega:\phi>t\}\lfloor \Omega\big) dt\ge \beta_L  c_* \mathcal{H}^2(\Gamma),
\end{align*}
which gives \eqref{ineq:energy-lead}.
\end{proof}
\subsection{Upper bound estimates}\label{subsec:upper-bound}
For the upper bound estimates, it suffices to construct a comparison map $\QQ_\ve$ which satisfies the boundary condition and closed to the upper bound. For this end, we can take
\begin{align*}
  \QQ_\ve=s_\ve(x)(\nn_\ve(x)\otimes\nn_\ve(x)-\frac13\II).
\end{align*}
Thus,
\begin{align*}
  \mathcal{F}_\ve(\QQ_\ve)=&\int_{\Omega}\Big\{\frac{L_1}2(|\nabla \QQ_\ve|^2+{L}_2|\nabla\cdot\QQ_\ve|^2)+\frac{1}{\ve^2}F_b(\QQ_\ve)\Big\}\ud x\\
  =&\int_{\Omega}\Big\{L_1(\frac13|\nabla s_\ve|^2+s_\ve^2|\nabla\nn_\ve|^2)+\frac{{L}_2}{2}\Big(\frac{1}{9}|\nabla s_\ve|^2+\frac{1}{3}(\nn_\ve\cdot\nabla s_\ve)^2+s_\ve^2(\nabla\cdot\nn_\ve)^2\\
  &\quad  +\frac{4}{3}s_\ve(\nn_\ve\cdot\nabla s_\ve) (\nabla\cdot\nn_\ve)-\frac{2}{3}s_\ve\nabla s_\ve\cdot(\nn_\ve\cdot\nabla )\nn_\ve\Big)
 +\frac{1}{\ve^2}F_b(s_\ve)\Big\}\ud x\\
  =&\int_{\Omega}\Big\{\frac13\Big(L_1+\frac{2}{3}L_2\Big)|\nabla s_\ve|^2+L_1s_\ve^2|\nabla\nn_\ve|^2-\frac{{L}_2}{6}|\nn_\ve\times\nabla s_\ve|^2 \\
  &\quad +\frac{L_2}{2}  \Big(s_\ve^2(\nabla\cdot\nn_\ve)^2+\frac{2}{3}s_\ve\nabla s_\ve\cdot(\nn_\ve\cdot\nabla )\nn_\ve \Big)\\
  &\quad +\frac{1}{\ve^2}F_b(s_\ve)+\frac{2L_2}{3}\nabla\cdot\Big(s_\ve^2\big[(\nabla\cdot \nn_\ve)\nn_\ve-(\nn_\ve\cdot\nabla)\nn_\ve\big]\Big)\Big\}\ud x\\
  \triangleq &W_\ve(s_\ve, \nn_\ve).
\end{align*}
Then from the fact that $-\frac{3}{2}L_1<L_2<0$, we know that $W_\ve(s_\ve, \nn_\ve)$ meet the requirements of case $(**)$  in \cite[Proposition 1.1]{LW23} on the coefficients. Therefore, the upper bound estimate directly follows from the upper bound estimate of case ({\bf B}) in \cite[Section 2.2]{LW23}.

\bigskip

\section{Asymptotics for the next order}
This section is devoted to the proof of Theorem \ref{thm:main2}.

Let $d_*=\frac12d(\mathcal{N},0)$.  For $\delta\in(0,d_*)$, define
  \begin{align*}
  \Omega^+_{\ve,\delta}=\{x\in\Omega:\quad d_\mathcal{N}(\QQ_\ve(x))\le \delta\},\quad E^+_{\ve,\delta}=\{x\in\Omega:\quad d_\mathcal{N}(\QQ_\ve(x))> \delta, |\QQ_\ve|> \delta\}.
  \end{align*}
Then for $x\in\Omega_{\ve,\delta}^+$, we have $\QQ_\ve\in N(\mathcal{N}, \delta)$(the $\delta$-neighbourhood of $\mathcal{N}$) which can be written as
\begin{align*}
  \QQ_\ve(x)= \bar{\QQ}_\ve(x)+\rho_\ve(x) \vv_\ve(x),\quad\text{ with }\bar{\QQ}_\ve(x)\in\mathcal{N},~ \rho_\ve=d_\mathcal{N}(\QQ_\ve),~ \vv_\ve=\frac{\QQ_\ve-\bar{\QQ}_\ve}{|\QQ_\ve-\bar{\QQ}_\ve|}.
\end{align*}
\subsection{Refined lower bound estimate}
\begin{proposition}
There exists $C>0$ independent of $\ve$ such that
\begin{align}\nonumber
   &\int_{\Omega}\Big(\frac{\beta_L^2}2|\nabla\QQ_\ve|^2+\frac{1}{\ve^2}F_b(\QQ_\ve)\Big)\ud x\\ \label{prop:decomp}
  &\ge \frac{\beta_L^2}2(1-C \delta)\int_{\Omega^+_{\ve,\delta}} |\nabla \bar{\QQ}_\ve(x)|^2\ud x+ \frac{\beta_Lc_*\mathcal{H}(\Gamma)}{\ve}-C.
\end{align}
\end{proposition}
\begin{proof}
Direct calculation gives us that
\begin{align*}
  |\nabla\QQ_\ve|^2=&|\nabla \bar{\QQ}_\ve(x)|^2+|\nabla \rho_\ve|^2+\rho_\ve^2|\nabla\vv_\ve|^2 +2\rho_\ve\nabla \bar{\QQ}_\ve(x):\nabla\vv_\ve\\
  =&|\nabla \bar{\QQ}_\ve(x)|^2+|\nabla \rho_\ve|^2+\rho_\ve^2|\nabla\vv_\ve|^2 -2\rho_\ve A(\bar{\QQ}_\ve(x))( \nabla \bar{\QQ}_\ve(x),\nabla\bar{\QQ}_\ve(x))\\
  \ge& |\nabla \rho_\ve|^2+(1-C \delta) |\nabla \bar{\QQ}_\ve(x)|^2.
\end{align*}
where $A(\bar\QQ_\ve)(\nabla\bar\QQ_\ve, \nabla\bar\QQ_\ve)$ is the second fundamental form of $\mathcal{N}$.
Thus
\begin{align} \nonumber
  &\int_{\Omega^+_{\ve,\delta}}\Big(\frac{\beta_L^2}2|\nabla\QQ_\ve|^2+\frac{1}{\ve^2}F_b(\QQ_\ve)\Big)\ud x\\ \label{decomp-0}
  \ge &\int_{\Omega^+_{\ve,\delta}}\frac{\beta_L^2}2(1-C \delta) |\nabla \bar{\QQ}_\ve(x)|^2\ud x+  \int_{\Omega^+_{\ve,\delta}}\Big(\frac{\beta_L^2}2|\nabla \rho_\ve|^2+\frac{1}{\ve^2}F_b(\QQ_\ve)\Big)\ud x.
\end{align}
By Lemma \ref{lem:bulk-ineq} and co-area formula, we have
\begin{align}\nonumber
& \int_{\Omega^+_{\ve,\delta}}\Big(\frac{\beta_L^2}2|\nabla \rho_\ve|^2+\frac{1}{\ve^2}F_b(\QQ_\ve)\Big)\ud x\\\nonumber
    \ge& \int_{\Omega^+_{\ve,\delta}}\Big(\frac{\beta_L^2}2|\nabla \rho_\ve|^2+\frac{1}{\ve^2}f(\rho_\ve)\Big)\ud x\\\nonumber
    \ge& \frac{\beta_L}{\ve} \int_{\Omega^+_{\ve,\delta}}|\nabla \rho_\ve|\sqrt{2f(\rho_\ve)}\ud x\\
    \ge& \frac{\beta_L}{\ve} \int_{0}^{\delta}\sqrt{2f(s)}\mathcal{H}^2(x: d_\mathcal{N}(\QQ_\ve(x))=s)\ud s. \label{decomp-1}
\end{align}
Similarly, we have
\begin{align}\nonumber
  &\int_{E_{\ve,\delta}^+}\Big(\frac{\beta_L^2}2|\nabla\QQ_\ve|^2+\frac{1}{\ve^2}F_b(\QQ_\ve)\Big)\ud x\\\nonumber
 & \ge \frac{\beta_L}{\ve}  \int_{E_{\ve,\delta}^+} |\nabla {\QQ}_\ve(x)|\sqrt{2F_b(\QQ_\ve)}\ud x\\\nonumber
&  \ge  \frac{\beta_L}{\ve}  \int_{E_{\ve,\delta}^+\cap\{d_\mathcal{N}(\QQ_\ve)\le d_*\}} |\nabla d_\mathcal{N}(\QQ_\ve)\sqrt{2f(d_\mathcal{N}(\QQ_\ve))}|\ud x\\
&  =  \frac{\beta_L}{\ve} \int_{\delta}^{d_*}\sqrt{2f(s)}\mathcal{H}^2(x: d_\mathcal{N}(\QQ_\ve(x))=s)\ud s, \label{decomp-2}
\end{align}
and
\begin{align}\nonumber
  &\int_{|\QQ_\ve|<d_*}\Big(\frac{\beta_L^2}2|\nabla\QQ_\ve|^2+\frac{1}{\ve^2}F_b(\QQ_\ve)\Big)\ud x\\\nonumber
  \ge &  \int_{|\QQ_\ve|<d_*}\Big(\frac{\beta_L^2}2|\nabla \rho_\ve|^2+\frac{1}{\ve^2}F_b(\rho_\ve)\Big)\ud x\\\nonumber
  \ge & \frac{\beta_L}{\ve} \int_{E_{\ve,\delta}^-} |\nabla |\QQ_\ve|\sqrt{2f(|\QQ_\ve|)}|\ud x\\
  = & \frac{\beta_L}{\ve} \int_{0}^{d_*}\sqrt{2f(s)}\mathcal{H}^2(x: |\QQ_\ve(x)|=s)\ud s. \label{decomp-3}
      \end{align}
Combing (\ref{decomp-0})-(\ref{decomp-3}) and noticing that $\{|\QQ_\ve|<d_*\}\cap \{d_\mathcal{N}(\QQ_\ve)<d_*\}=\emptyset$, we arrive
\begin{align}\nonumber
  &\int_{\Omega}\Big(\frac{\beta_L^2}2|\nabla\QQ_\ve|^2+\frac{1}{\ve^2}F_b(\QQ_\ve)\Big)\ud x\\ \nonumber
  &\ge \int_{\Omega^+_{\ve,\delta}}\frac{\beta_L^2}2(1-C \delta) |\nabla \bar{\QQ}_\ve(x)|^2\ud x\\ \label{decomp}
  &\quad +  \frac{\beta_L}{\ve}\int_{0}^{d_*}\sqrt{2f(s)}\Big\{\mathcal{H}^2(x: |\QQ_\ve(x)|=s)+\mathcal{H}^2(x: d_\mathcal{N}(\QQ_\ve(x))=s)\Big\}\ud s.
\end{align}
From the assumption on boundary data,  we have
\begin{align*}
  \mathcal{H}^2(\Sigma\cap \{|\QQ_\ve|\ge C\ve \}\cap  \{d_\mathcal{N}(\QQ_\ve)\ge C \ve\}) \le C\ve.
\end{align*}
Thus the area minimality of $\Gamma$ gives us that for $s\in (C\ve, c_*)$,
\begin{align*}
\mathcal{H}^2(\Gamma)&\le \mathcal{H}^2(x: |\QQ_\ve(x)|=s)+
 \mathcal{H}^2(\Sigma\cap \{|\QQ_\ve|\ge C\ve \}\cap  \{d_\mathcal{N}(\QQ_\ve)\ge C \ve\}) \le C\ve,\\
\mathcal{H}^2(\Gamma)&\le \mathcal{H}^2(x: d_\mathcal{N}(\QQ_\ve(x))=s)+ \mathcal{H}^2(\Sigma\cap \{|\QQ_\ve|\ge C\ve \}\cap  \{d_\mathcal{N}(\QQ_\ve)\ge C \ve\}) \le C\ve.
\end{align*}
Thus
\begin{align*}
  &\int_{0}^{d_*}\sqrt{2f(s)}(\mathcal{H}^2(x: |\QQ_\ve(x)|=s)+\mathcal{H}^2(x: d_\mathcal{N}(\QQ_\ve(x))=s))\ud s\\
  &\ge \int_{C\ve}^{d_*}\sqrt{2f(s)}(\mathcal{H}^2(x: |\QQ_\ve(x)|=s)+\mathcal{H}^2(x: d_\mathcal{N}(\QQ_\ve(x))=s))\ud s\\
 & \ge 2(\mathcal{H}^2(\Gamma)-C\ve) \int_{C\ve}^{d_*}\sqrt{2f(s)}\ud s\\
 &\ge c_*\mathcal{H}^2(\Gamma) -O(\ve).
\end{align*}
which concludes the proof.
\end{proof}

For $t$, define
\begin{align*}
  \Gamma(t)=&\{x\in\Omega: \quad d_\Gamma(x)=t\}\\
   \Sigma(t)=&\{x\in\partial\Omega: \quad d_{\Gamma}(x)=t\}~(=\partial\Gamma(t)).
\end{align*}
Let $\mathcal{S}(t)$ be a minimal hypersurface spanned by $\Sigma(t)$:
\begin{align*}
  \mathcal{H}^2(\mathcal{S}(t))=\min\{\mathcal{H}^{2}(S): S \text{ is an integral 2-current in $\Omega$ with $\partial S=\Sigma(t)$}\}.
\end{align*}

Assume that
\begin{align}\label{energy:refine-bound}
  \mathcal{F}_\ve(\QQ_\ve; \Omega)-\frac{c_*\beta_L}{\ve}\mathcal{H}^{2}(\Gamma)\le C.
\end{align}
Now we define
\begin{align*}
\tilde{\QQ}_{\ve}=\bar{\QQ}_\ve\chi_{\Omega^+_{\ve,\delta_\ve}},
\end{align*}
where $\delta_\ve$  is chosen as in Section 3 such that \eqref{estimate:boundary} is satisfied.
We will show that $\{\tilde{\QQ}_{\ve}\}$ is bounded in $\mathrm{SBV}(\Omega)$.
\begin{proposition}
Assume $\{\QQ_\ve\}$ satisfies \eqref{energy:refine-bound}, then $\{\tilde\QQ_\ve\}$ is uniformly bounded in $\mathrm{SBV}(\Omega)$. In addition, there is $\QQ_*\in \mathrm{SBV}(\Omega,\mathcal{Q})$ such that, after passage to subsequences,
\begin{align}\label{converge:L1}
 \tilde{\QQ}_{\ve}&\to \QQ_*\chi_{\Omega_+} \text{ strongly in } L^1, \\
  \nabla\tilde{\QQ}_{\ve}&\to \nabla\QQ_*\chi_{\Omega_+} \text{ weakly in } L^1. \label{converge:weakL1}
\end{align}
Here $ \nabla\tilde{\QQ}_{\ve}$ is the absolute continuous part of the distributional derivative of $\tilde{\QQ}_{\ve}$.
\end{proposition}
\begin{proof}
By (\ref{prop:decomp}), the absolute continuous part of its distributional derivative
\begin{align*}
\int_{\Omega}|\nabla \tilde{\QQ}_\ve|^2\ud x=\int_{\Omega^+_{\ve,\delta_\ve}}|\nabla \bar{\QQ}_\ve|^2\ud x\le C,\quad
\int_{\Omega}|\tilde{\QQ}_\ve|\ud x=\int_{\Omega^+_{\ve,\delta_\ve}}|\bar{\QQ}_\ve|\ud x\le C.
\end{align*}
Moreover, the set of discontinuities of $\tilde{\QQ}_{\ve}$, denoted by $J_{\tilde{\QQ}_{\ve}}$, satisfies
\begin{align*}
\mathcal{H}^2(J_{\tilde{\QQ}_{\ve}})\le  \mathcal{H}^2(\partial\Omega^+_{\ve,\delta}\cap \Omega )\le \mathcal{H}^2(\Gamma)+1.
\end{align*}
Therefore, $\{\tilde{\QQ}_{\ve}\}$ is uniformly bounded in $\mathrm{SBV}(\Omega)$. Thus by weak compactness, we may assume
there exists $\QQ_*\in \mathrm{SBV}(\Omega)$ such that \eqref{converge:L1}-\eqref{converge:weakL1} hold.

Moreover, as $\bar{\QQ}_\ve\in \mathcal{N}$ a.e. $x\in\Omega^+_{\ve,\delta_\ve}$
and $\tilde{\QQ}_\ve=0$ in $\Omega\setminus\Omega^+_{\ve,\delta_\ve}$, we have by (\ref{converge:L1}) that $\QQ_*\in \mathcal{N}$ a.e. $x\in\Omega_+$ and  $\QQ_*=0$ a. e. $x\in\Omega_-$.
\end{proof}

\begin{Lemma} $\QQ_\ve\to \QQ_*$ strongly in $L^1(\Omega)$.
\end{Lemma}
\begin{proof}

Due to \eqref{converge:L1}, it suffices to prove $|{\QQ}_\ve-\tilde{\QQ}_\ve|_{L^1(\Omega)}\to 0$, which is a consequence of
 the following estimates:
\begin{align}
\int_{\Omega^+_{\ve,\delta}}  |\QQ_\ve-\bar{\QQ}_\ve| \ud x  \to 0,\quad
\int_{\Omega\setminus\Omega^+_{\ve,\delta_\ve}}|\QQ_\ve|\ud x \to 0,\quad \text{as }\ve\to0.
\end{align}
The first one can be deduced as
\begin{align*}
&  \int_{\Omega^+_{\ve,\delta_\ve}}  |\QQ_\ve-\bar{\QQ}_\ve| \ud x = \int_{\Omega^+_{\ve,\delta_\ve}}  d_\mathcal{N}(\QQ_\ve) \ud x \\
 & \le
  C\int_{\Omega^+_{\ve,\delta_\ve}}  \sqrt{F_b(\QQ_\ve)} \ud x\le C|\Omega|^{1/2} \Big(\int_{\Omega}  {F_b(\QQ_\ve)} \ud x\Big)^{1/2}\to 0.
\end{align*}
For the second one, we have $d_\mathcal{N}(\QQ_\ve)\ge \delta_\ve $ for $x\in \Omega\setminus\Omega^+_{\ve,\delta_\ve}$. Thus using Lemma \ref{lem:bulk-ineq}-(iii), one get
\begin{align*}
 \int_{\Omega\setminus\Omega^+_{\ve,\delta_\ve}}|\QQ_\ve|\ud x
 \le \frac{C}{\delta_\ve} \int_{\Omega\setminus\Omega^+_{\ve,\delta_\ve}}|\sqrt{F_b(\QQ_\ve)}|\ud x
 \le \frac{C}{\delta_\ve}|\Omega|^{1/2} \int_{\Omega\setminus\Omega^+_{\ve,\delta_\ve}}F_b(\QQ_\ve)\ud x\le  \frac{C}{\delta_\ve}\ve^{1/2}\to 0,
\end{align*}
 as $\ve\to 0$.
\end{proof}

\begin{proposition}
The limit map $\QQ_*$ satisfies:
\begin{align*}
\QQ_*\in H^1(\Omega^+, \mathcal{Q}),\quad \QQ_*|_{\partial\Omega^+\cap \Sigma}=G,\\
\QQ_*\barwedge\nu_{\Gamma}=0,\quad\text{ a.e. on }\Gamma.
\end{align*}
Moreover
\begin{align*}
&\int_{\Omega_+}\bigg\{\frac{\beta_L^2}{2}|\nabla\QQ_*|^2-\frac{{L}_2}{12}|\mathcal{T}(\QQ_*)|^2\bigg\}\ud x\\
&\le  \liminf_{\ve\to 0}\Big\{\int_{\Omega}\Big(\frac{\beta_L^2}{2}|\nabla\QQ_\ve|^2+\frac{1}{\ve^2}F_b(\QQ_\ve)
-\frac{{L}_2}{12}|\mathcal{T}(\QQ_\ve)|^2 \Big)\ud x-\frac{c_*\beta_L\mathcal{H}^2(\Gamma)}{\ve}\Big\}.
\end{align*}
\end{proposition}
\begin{proof}
By the lower semi-continuity, we have
\begin{align*}
\int_{\Omega_+}|\nabla\QQ_*|^2\ud x\le \liminf_{\ve\to 0}\int_{\Omega_+}|\nabla \bar{\QQ}_\ve|^2\ud x
\le \liminf_{\ve\to 0}\int_{\Omega^+_{\ve,\delta}}|\nabla \bar{\QQ}_\ve|^2\ud x.
\end{align*}
To show the zero $\mathcal{H}^2$-measure of its jump set, we have
\begin{align*}
\mathcal{H}^2(J_{\QQ_*}\cap \{x\in\Omega^+: d_{\Gamma}(x)>\delta_\ve \})\le& \liminf_{\ve\to 0}
\mathcal{H}^2(J_{\bar{\QQ}_\ve}\cap \{x\in\Omega^+: d_{\Gamma}(x)>\delta_\ve \})\\
\le &\liminf_{\ve\to 0}\mathcal{H}^2(\partial\Omega_{\ve,\delta_\ve}\cap\{x\in\Omega^+: d_{\Gamma}(x)>\delta_\ve \})=0.
\end{align*}
Thus $\QQ_*\in H^1(\Omega^+,\mathcal{Q})$. From the boundary condition for $\QQ_\ve$, it is straightforward to get that $\QQ_*=\mathbf{E}$ on $\Sigma_+$.
For $\forall \Phi_{ij}\in C_0^\infty(\Omega_+)$, we have
\begin{align*}
\int_{\Omega_+}(\mathcal{T}(\QQ_\ve)-\mathcal{T}(\QQ_*)):\Phi\ud x
=\int_{\Omega_+}(\QQ_\ve-\QQ_*):\mathcal{T}(\Phi)\ud x \to 0,\quad \text{as }\ve\to 0.
\end{align*}
Since
$\mathcal{T}(\QQ_\ve)$ and  $ \mathcal{T}(\QQ_*)$ are uniformly bounded in $L^2(\Omega_+)$, one has
\begin{align}\label{weak:TQ+}
  \mathcal{T}(\QQ_\ve)\rightharpoonup \mathcal{T}(\QQ_*)\quad \text{ weakly in }L^2(\Omega_+).
\end{align}
Similarly, from the fact that $\QQ_\ve\to 0$ in $L^1(\Omega_-)$, we can derive
\begin{align*}
  \mathcal{T}(\QQ_\ve)\rightharpoonup  0 \quad \text{ weakly in }L^2(\Omega_-).
\end{align*}
Therefore, we get
\begin{align*}
   \mathcal{T}(\QQ_\ve)\rightharpoonup  \mathcal{T}(\QQ_*)\chi_{\Omega_+} \quad \text{ weakly in }L^2(\Omega).
\end{align*}

For any $\Phi_{ij}\in C_0^\infty(\Omega)$, we have
\begin{align*}
 \int_{\Omega_+} \Phi: \mathcal{T}(\QQ_*)\ud x =  \lim_{\ve\to 0} \int_{\Omega} \Phi:\mathcal{T}(\QQ_\ve)\ud x
= \lim_{\ve\to 0} \int_{\Omega} \mathcal{T}(\Phi):\QQ_\ve\ud x
=\int_{\Omega_+} \mathcal{T}(\Phi):\QQ_*\ud x,
\end{align*}
which gives that
\begin{align*}
   \int_{\Gamma} \Phi:( \QQ_*\barwedge \nu_{\Gamma})\ud x=0.
\end{align*}
Thus $\QQ_*\barwedge \nu_{\Gamma}=0 $ a.e. on $\Gamma$.

In addition, by (\ref{weak:TQ+}) and the lower semicontinuity, we have
\begin{align*}
\int_{\Omega_+}|\mathcal{T}(\QQ_*)|^2\ud x\le \liminf_{\ve\to 0} \int_{\Omega_+} |\mathcal{T}(\QQ_\ve)|^2\ud x.
\end{align*}

In summary, we get
\begin{align*}
\liminf_{\ve\to 0}\int_{\Omega}\Big\{\frac{\beta_L^2}{2}|\nabla\QQ_\ve|^2+\frac{1}{\ve^2}F_b(\QQ_\ve)
-\frac{{L}_2}{12}|\mathcal{T}(\QQ_\ve)|^2\Big\} \ud x-\frac{c_*\beta_L\mathcal{H}^2(\Gamma)}{\ve}\\
\ge \int_{\Omega_+}\Big\{\frac{\beta_L^2}{2}|\nabla\QQ_*|^2-\frac{{L}_2}{12}|\mathcal{T}(\QQ_*)|^2\Big\}\ud x.
\end{align*}

\end{proof}

\subsection{Refined upper bound estimate}
Similar to Section \ref{subsec:upper-bound}, the refined upper bound estimates directly follows from the corresponding estimates in \cite[Lemma 3.2]{LW23}.

\section{Appendix}
Let $\beta=\beta_L>0$ and $R>0$ large. Take $\sigma>0$, and  $\phi_\sigma$ :$(-\infty,+\infty)\to\mathbb{R}$ being the solution of
\begin{align*}
  \beta\phi'(z)=\phi(1-\phi)+\sigma, \text{ with }\phi(0)=1/2.
\end{align*}
Then
\begin{align*}
  \phi_\sigma(z)= \frac{1 + \sqrt{1+4\sigma} + \left(1 - \sqrt{1+4\sigma}\right) e^{-z\sqrt{1+4\sigma}/\beta}}{2 \left(1 + e^{-z\sqrt{1+4\sigma}/\beta}\right)}
\end{align*}
Choose $\sigma$ such that $\phi_\sigma(R)=1$, or equivalently,
\begin{align}\label{relation:sigma-R}
  R=\frac{\beta}{\sqrt{1+4\sigma}}\ln\frac{\sqrt{1+4\sigma}+1}{\sqrt{1+4\sigma}-1}.
\end{align}
Then we have
\begin{align*}
  \sigma=2e^{-R/\beta}(1+O({R}^{-1})).
\end{align*}
Let $S_R(z)=\phi_\sigma(z)$ for $z\in[-R, R]$ with \eqref{relation:sigma-R}. Then $S_R(-z)=1-S_R(z)$ for $|z|\le R$ and
\begin{align*}
 & S_R(R)=1,\quad  S_R(-R)=0,\\
 & \max_{|z|\le R}\big|\beta^2(S_R')^2+S_R^2(1-S_R)^2-2\beta\sqrt{S_R^2(1-S_R)^2}|S_R'|\big|\\
  &\qquad =\max_{|z|\le R}(\beta S_R'-S_R(1-S_R))^2=\sigma^2\le Ce^{-2R/\beta},\\
  &\int_{-R}^R \big(\beta^2(S_R')^2+S_R^2(1-S_R)^2\big)\ud z\le \int_{-R}^R 2\beta S_R'S_R(1-S_R)\ud z+CRe^{-2R/\beta}=\frac{\beta}3+Ce^{-R/\beta},\\
  & |S_R'(z)|\le 1-S_R+\sigma\le C(e^{-z/\beta}+\sigma)\le C e^{-z/\beta},\text{ for }|z|\le R.
  \end{align*}
Therefore, $S_R$ meets all the requirements.

\section*{Acknowledgments}
This work is supported by National Key Research and Development Program of China (No. 2023YFA1008801).
F.H. Lin is supported by an NSF grant DMS 2247773. W. Wang is supported by NSF of China under Grant No. 12271476.

\end{document}